\documentclass[10pt,oneside,reqno]{amsart}

\usepackage[a4paper,margin=1in]{geometry}

\pdfoutput=1

\usepackage{amsmath}
\usepackage{amssymb}
\usepackage{amsthm}
\usepackage{mathtools}
\usepackage{enumitem}
\usepackage{tikz-cd}

\usepackage[
backend=biber,
style=alphabetic,
citestyle=alphabetic,
giveninits=true,
doi=false,
isbn=false,
url=false,
maxnames=50,
maxalphanames=7
]{biblatex}

\DeclareNameAlias{default}{given-family}
\DeclareDelimFormat[bib]{nametitledelim}{\addcomma\space}
\DeclareDelimFormat{finalnamedelim}{\addspace\bibstring{and}\space}
\DeclareFieldFormat[article,incollection,unpublished]{title}{\mkbibemph{#1}}
\DeclareFieldFormat{journaltitle}{#1\isdot}
\DeclareFieldFormat[article]{volume}{\mkbibbold{#1}}
\DeclareFieldFormat[article]{number}{\mkbibparens{#1}}
\DeclareFieldFormat[article]{pages}{#1}
\renewbibmacro*{in:}{\ifentrytype{article}{}{\printtext{\bibstring{in}\intitlepunct}}}
\renewbibmacro*{volume+number+eid}{\printfield{volume}\setunit*{}\printfield{number}\setunit{\bibeidpunct}\printfield{eid}}
\renewbibmacro*{note+pages}{\printfield{pages}\setunit{\addperiod\space}\printfield{note}\newunit}

\usepackage[linktocpage=true]{hyperref}

\mathtoolsset{showonlyrefs,showmanualtags}

\allowdisplaybreaks

\newtheorem{theorem}{Theorem}
\newtheorem{proposition}[theorem]{Proposition}
\newtheorem{corollary}[theorem]{Corollary}
\newtheorem{lemma}[theorem]{Lemma}

\theoremstyle{definition}
\newtheorem{definition}[theorem]{Definition}

\numberwithin{theorem}{section}
\numberwithin{equation}{section}

\title{Trace radicals and cocenters of free products}

\date{}

\author{Nikita Safonkin}

\address{Institute of Mathematics, Leipzig University, Augustusplatz 10, 04109 Leipzig, Germany.}

\email{safonkin.nik@gmail.com}

\begin{document}

\begin{abstract}
    We call a unital associative algebra \(A\) trace residually finite-dimensional if its elements are separated by finite-dimensional representations and its cocenter \(A/[A,A]\) is separated by the corresponding trace functionals.
    For RFD algebras \(A\) and \(B\), we prove that the trace radical of \(A*B\), the common kernel of all finite-dimensional trace functionals, is the direct sum of the trace radicals of \(A\) and \(B\), which implies that \(A*B\) is trace RFD if and only if both \(A\) and \(B\) are trace RFD.
    The proof combines an explicit cocenter decomposition with a construction of finite-dimensional representations whose traces detect nonzero classes of words of length greater than one.
\end{abstract}

\maketitle

\section{Introduction}

Fix a field \(\Bbbk\) of characteristic zero. Throughout, \emph{algebra} means a nonzero associative unital \(\Bbbk\)-algebra, and a \emph{finite-dimensional representation} of an algebra \(A\) is a unital homomorphism \(\rho\colon A\to\operatorname{End}_{\Bbbk}(V)\) with a nonzero finite-dimensional \(V\).
An algebra \(A\) is \emph{residually finite-dimensional} (RFD) if for every nonzero \(a\in A\) there is a finite-dimensional representation \(\rho\) such that \(\rho(a)\neq0\). Equivalently, every nonzero element of \(A\) survives under a homomorphism from \(A\) to a finite-dimensional \(\Bbbk\)-algebra.

Write \(A_{\natural}:=A/[A,A]\) for the \emph{cocenter} of \(A\), where \([A,A]\) is the \(\Bbbk\)-linear span of the commutators \(ab-ba\) with \(a,b\in A\), and let \(\overline{a}\in A_{\natural}\) denote the class of \(a\).
Since the matrix trace vanishes on commutators, every finite-dimensional representation \(\rho\) of \(A\) induces a linear functional
\begin{equation}
    \operatorname{tr}(-)(\rho)\colon A_{\natural}\to\Bbbk,\qquad \operatorname{tr}(\overline{a})(\rho):=\operatorname{tr}(\rho(a)).
\end{equation}

\begin{definition}\label{def_trace_rfd}
    An algebra \(A\) is \emph{trace residually finite-dimensional} (trace-RFD) if it is RFD and, in addition, the trace functionals separate the elements of \(A_{\natural}\): for every nonzero \(\overline{a}\in A_{\natural}\) there is a finite-dimensional representation \(\rho\) with \(\operatorname{tr}(\overline{a})(\rho)\neq0\).
\end{definition}

Trace separation alone does not imply RFD in general. For a free product \(A*B\), however, trace separation implies RFD whenever neither factor is the base field; see Corollary~\ref{cor_trace_separation_upgrade} below.

For free associative algebras, the tracial Nullstellensatz of Klep and \v{S}penko \cite[Theorem~3.1]{KlepSpenko2014} implies that, over an algebraically closed field of characteristic zero, traces separate the cocenter.
Since free associative algebras are RFD, they are therefore trace-RFD.

The trace-separation condition also has a standard representation-theoretic formulation.
Let \(A\text{-}\mathrm{mod}_{\mathrm{fd}}\) be the category of finite-dimensional left \(A\)-modules, let \(K_0\bigl(A\text{-}\mathrm{mod}_{\mathrm{fd}}\bigr)\) be its Grothendieck group, and set
\(
    R(A):=
    K_0\bigl(A\text{-}\mathrm{mod}_{\mathrm{fd}}\bigr)
    \otimes_{\mathbb{Z}}\Bbbk,
\)
with dual vector space \(R(A)^*\).
For \(a\in A\) and \(M\in A\text{-}\mathrm{mod}_{\mathrm{fd}}\), write \(a_M\in\operatorname{End}_{\Bbbk}(M)\) for the action of \(a\) on \(M\).
There is a natural map
\[
    A_{\natural}\longrightarrow R(A)^*,
    \qquad
    \overline{a}\longmapsto
    \bigl([M]\longmapsto\operatorname{tr}(a_M)\bigr)
\]
and trace separation means precisely that this map is injective.
For \(A=\Bbbk[G]\), where \(G\) is a finite group and \(\Bbbk\) is algebraically closed, it is given by the ordinary character table and is an isomorphism.
Analogous maps have been studied for Hecke algebras in \cite{Kazhdan1986Cuspidal},\cite{BernsteinDeligneKazhdan1986}, \cite[Theorem~B]{Kazhdan1986CloseFields}, \cite{CiubotaruHe2016},\cite{CiubotaruHe2017}, \cite{HeKim2019}.

The trace-RFD condition has a noncommutative-geometric interpretation through the Kontsevich--Rosenberg principle \cite{KontsevichRosenberg2000}.
This principle proposes that a ''geometric'' structure on an associative algebra \(A\) should induce the corresponding ''commutative'' structure on the representation schemes \(\operatorname{Rep}_N(A)\) parametrizing \(N\)-dimensional representations of \(A\), for every \(N\geq1\).
Ginzburg and Schedler recast this principle in terms of a commutative algebra \(\mathcal{F}(A):=\operatorname{Sym}_{\mathbb{S}}\bigl(A_{\natural}\oplus A[-1]\bigr)\) in the symmetric monoidal category of diagonal \(\mathbb{S}\)-bimodules \cite{GinzburgSchedler2010}.
Interpreting \(\mathcal{F}(A)\) as the algebra of functions on the ``noncommutative scheme associated with \(A\)'', we should then formulate the property of being reduced as the vanishing of a certain ideal in \(\mathcal{F}(A)\) \cite{Safonkin2025DoubleStarProduct}.
This ideal vanishes precisely when \(A\) is trace-RFD.

Recall that the cocenter of an algebra naturally identifies with its zeroth cyclic homology.
This identification allows one to describe the cocenter of the free product of two augmented algebras by specializing a theorem of Feigin and Tsygan \cite[Theorem~3.2.1]{FeiginTsygan1987} about the cyclic homology of the free product to degree zero; Ginzburg and Schedler obtain a related description for \(B*\Bbbk[t]\) in \cite[Section~3.1]{GinzburgSchedler2012}.
For arbitrary unital algebras, which need not admit augmentations, Proposition~\ref{prop_free_product_cocenter_decomposition} below provides an analogous description.

Define the \emph{trace radical} of a \(\Bbbk\)-algebra \(C\) by
\[
    \operatorname{Rad}_{\mathrm{tr}}(C):=
    \bigcap_{\rho}
    \ker\bigl(\operatorname{tr}(-)(\rho)\bigr)
    \subseteq C_{\natural},
\]
where \(\rho\) ranges over the finite-dimensional representations of \(C\).

Based on the description of the cocenter, our main result gives an exact formula for the trace radical of a free product.

\begin{theorem}\label{thm_trace_radical_free_product}
    Let \(A\) and \(B\) be RFD \(\Bbbk\)-algebras, and let \(\iota_A\colon A_{\natural}\to(A*B)_{\natural}\) and \(\iota_B\colon B_{\natural}\to(A*B)_{\natural}\) be the maps induced by the canonical homomorphisms. Then \(\iota_A\) and \(\iota_B\) are injective, and
    \[
        \operatorname{Rad}_{\mathrm{tr}}(A*B)
        =\iota_A\bigl(\operatorname{Rad}_{\mathrm{tr}}(A)\bigr)
        \oplus\iota_B\bigl(\operatorname{Rad}_{\mathrm{tr}}(B)\bigr).
    \]
\end{theorem}

It is known that the free product of two algebras is RFD if and only if both factors are RFD \cite[Theorem~2]{lichtman1983residual}.
Combining this result with Theorem~\ref{thm_trace_radical_free_product}, we obtain the following corollary.

\begin{corollary}\label{cor_trace_rfd_free_product}
    Let \(A\) and \(B\) be associative \(\Bbbk\)-algebras. Then \(A*B\) is trace-RFD if and only if both \(A\) and \(B\) are trace-RFD.
\end{corollary}

In the case of group algebras, our results admit a concrete group-theoretic interpretation. A group is \emph{residually finite} if every nontrivial element has nontrivial image in some finite quotient and \emph{conjugacy separable} if every pair of non-conjugate elements has non-conjugate images in a finite quotient.
Both properties are preserved under free products: residual finiteness by \cite{Gruenberg1957} and conjugacy separability by \cite{Stebe1970,Remeslennikov1971}.
Proposition~\ref{prop_group_algebra} below shows that residual finiteness of \(G\) implies that \(\Bbbk[G]\) is RFD, while over an algebraically closed field conjugacy separability implies that \(\Bbbk[G]\) is trace-RFD.
The isomorphism \(\Bbbk[G]*\Bbbk[H]\cong\Bbbk[G*H]\) shows that Corollary~\ref{cor_trace_rfd_free_product} is an analogue, for arbitrary associative algebras, of these group-theoretic preservation results.

The following application gives a concrete criterion for free products of finite-dimensional algebras.

\begin{corollary}\label{cor_finite_dimensional_free_products}
    Suppose that \(\Bbbk\) is algebraically closed, and let \(A_1,\ldots,A_m\) be finite-dimensional \(\Bbbk\)-algebras.
    For each \(i\), let \(J(A_i)\) denote the Jacobson radical of \(A_i\), let \(\overline{J(A_i)}\subseteq(A_i)_{\natural}\) denote its image in the cocenter, and let \(\iota_i\colon(A_i)_{\natural}\to(A_1*\cdots*A_m)_{\natural}\) be the map induced by the canonical homomorphism.
    Then
    \[
        \operatorname{Rad}_{\mathrm{tr}}(A_1*\cdots*A_m)
        =
        \bigoplus_{i=1}^{m}
        \iota_i\bigl(\overline{J(A_i)}\bigr).
    \]
    Consequently,
    \[
        A_1*\cdots*A_m\text{ is trace-RFD}
        \quad\Longleftrightarrow\quad
        J(A_i)\subseteq[A_i,A_i]\text{ for every }i.
    \]
\end{corollary}

In particular, the free products of finite-dimensional semisimple algebras are trace-RFD. These free products were recently studied via quivers in \cite{BuchananDimitrovGracePaquetteWehlauXu2024}.

\medskip

\textbf{Acknowledgments.}
The author used OpenAI's ChatGPT 5.5 to assist in developing the proof of Proposition~\ref{prop_cyclic_trace_detection}, including Lemma~\ref{lemma_two_algebra_coefficients_cyclic}.
The author independently verified the resulting arguments and takes full responsibility for the mathematical content of the paper.
The author was partially supported by the European Research Council (ERC), Grant Agreement No.~101041499.

\section{Cocenters of free products}\label{sec:cocenter_free_product}

Let \(A\) and \(B\) be \(\Bbbk\)-algebras. Throughout this section we fix vector-space decompositions
\begin{equation}
    A=\Bbbk 1\oplus A_0,\qquad B=\Bbbk 1\oplus B_0.
\end{equation}
Thus \(A_0\simeq A/\Bbbk\) and \(B_0\simeq B/\Bbbk \) as vector spaces.

We write \(A*B\) for the free product of \(A\) and \(B\) as unital algebras --- their coproduct in the category of unital \(\Bbbk\)-algebras. It is equipped with unit-preserving homomorphisms \(A\to A*B\) and \(B\to A*B\), under which \(1_A\) and \(1_B\) acquire the common image \(1_{A*B}\), and is universal among unital \(\Bbbk\)-algebras admitting such a pair of homomorphisms.

Recall that \(A*B=\operatorname{T}(A\oplus B)/J\) is the quotient of the tensor algebra on the vector space \(A\oplus B\) by the two-sided ideal \(J\) generated by \(a\otimes a'-aa'\) and \(b\otimes b'-bb'\) (\(a,a'\in A\), \(b,b'\in B\)) together with \(1_A-1\) and \(1_B-1\), where \(1\) is the unit of \(\operatorname{T}(A\oplus B)\). The canonical maps \(A,B\to A*B\) are injective, and, writing \(C_A:=A_0\) and \(C_B:=B_0\), the free product decomposes as a vector space:
\begin{equation}\label{f_free_product_normal_form}
    A*B=\Bbbk 1\oplus
    \bigoplus_{m\geq 1}\;
    \bigoplus_{\substack{\varepsilon_1,\ldots,\varepsilon_m\in\{A,B\}\\ \varepsilon_i\neq\varepsilon_{i+1}}}
    C_{\varepsilon_1}\otimes\cdots\otimes C_{\varepsilon_m}.
\end{equation}
An element of a summand \(C_{\varepsilon_1}\otimes\cdots\otimes C_{\varepsilon_m}\) of~\eqref{f_free_product_normal_form} is \emph{homogeneous of length \(m\)}, while a nonzero scalar is homogeneous of length \(0\). The \emph{length} of an element \(f\in A*B\) is the largest \(m\) for which the homogeneous length-\(m\) component of \(f\) is nonzero, and that component is the \emph{top component} of \(f\). A product \(c_1\cdots c_m\) whose letters \(c_p\) are drawn alternately from \(A_0\) and \(B_0\) is a \emph{reduced word} of length \(m\), and the reduced words of length \(m\) span the homogeneous part of length \(m\). For each \(m\geq1\) there are exactly two summands of length \(m\), distinguished by the first tensor factor. In particular the homogeneous part of even length \(2r\) is \((A_0\otimes B_0)^{\otimes r}\oplus(B_0\otimes A_0)^{\otimes r}\).

\begin{proposition}\label{prop_free_product_cocenter_decomposition}
    There is an isomorphism of vector spaces
    \begin{equation}\label{f_free_product_cocenter_decomposition}
        (A*B)_{\natural}
        \cong
        \frac{A_{\natural}\oplus B_{\natural}}
        {\Bbbk(\overline{1_A},-\overline{1_B})}
        \oplus
        \bigoplus_{r\geq1}
        \left((A_0\otimes B_0)^{\otimes r}\right)_{C_r},
    \end{equation}
    where \(C_r=\mathbb{Z}/r\mathbb{Z}\) cyclically permutes the \(r\) tensor factors \(A_0\otimes B_0\) and \((-)_{C_r}\) denotes coinvariants. Under this isomorphism, the class of \(a\in A\) (respectively, \(b\in B\)) in \((A*B)_{\natural}\) maps to the class of \((\overline a,0)\) (respectively, \((0,\overline b)\)) in the quotient on the right-hand side. For \(a_i\in A_0\) and \(b_i\in B_0\), the class of \(a_1b_1\cdots a_rb_r\) in \((A*B)_{\natural}\) maps to the coinvariant class of \((a_1\otimes b_1)\otimes\cdots\otimes(a_r\otimes b_r)\).
\end{proposition}

\begin{proof}
    Write
    \[
        P_{A,B}:=\frac{A_{\natural}\oplus B_{\natural}}{\Bbbk(\overline{1_A},-\overline{1_B})},
    \]
    and let \(q_A\colon A_{\natural}\to P_{A,B}\) and \(q_B\colon B_{\natural}\to P_{A,B}\) be the canonical maps.
    For \(r\geq1\), set \(W_r:=\bigl((A_0\otimes B_0)^{\otimes r}\bigr)_{C_r}\).
    Set \(D:=P_{A,B}\oplus\bigoplus_{r\geq1}W_r\).
    For \(w\in(A_0\otimes B_0)^{\otimes r}\), write \([w]_{C_r}\) for its class in \(W_r\).
    Define a linear map
    \[
        \jmath\colon D\longrightarrow(A*B)_{\natural}
    \]
    by
    \[
        \jmath\bigl(q_A(\overline a)\bigr):=\overline a,
        \qquad
        \jmath\bigl(q_B(\overline b)\bigr):=\overline b,
        \qquad
        \jmath\bigl([w]_{C_r}\bigr):=\overline w.
    \]
    The canonical homomorphisms \(A\to A*B\) and \(B\to A*B\) induce maps on the corresponding commutator quotients and send \(\overline{1_A}\) and \(\overline{1_B}\) to the same element of \((A*B)_{\natural}\), so the first two prescriptions are compatible with the quotient defining \(P_{A,B}\).
    The last prescription is well-defined because moving the last block \(a_r\otimes b_r\) to the front replaces \(a_1b_1\cdots a_rb_r\) by \(a_rb_ra_1b_1\cdots a_{r-1}b_{r-1}\), which has the same cocenter class.

    We first prove that \(\jmath\) is surjective.
    By~\eqref{f_free_product_normal_form}, the cocenter is spanned by the class of the unit and the classes of reduced words.
    The unit and the reduced words of length \(1\) lie in the image of \(\jmath\).
    Every reduced word of even length \(2r\) has the same cocenter class as a cyclic rotation beginning in \(A_0\), and hence its class lies in the image of \(\jmath\).
    Let \(u=c_1\cdots c_m\) be a reduced word of odd length \(m\geq3\), and write \(c_m c_1=\sigma 1+h\), where \(\sigma\in\Bbbk\) and \(h\in A_0\cup B_0\).
    Cyclicity in the commutator quotient gives
    \begin{equation}\label{f_free_product_odd_folding}
        \overline u
        =\sigma\,\overline{c_2\cdots c_{m-1}}
        +\overline{h c_2\cdots c_{m-1}}.
    \end{equation}
    The first term is represented by a reduced word of odd length \(m-2\), whereas the second term is either zero or represented by a reduced word of even length \(m-1\).
    Induction on \(m\) therefore shows that the class of every odd reduced word also lies in the image of \(\jmath\), proving surjectivity.

    We will need the following claim, which can be checked directly. Let \(X\) be a vector space and let \(T\colon A*B\to X\) be linear.
    Suppose that the restrictions of \(T\) to \(A\) and \(B\) vanish on commutators and that
    \begin{equation}\label{f_free_product_cyclic_invariance}
        T(c_1\cdots c_m)=T(c_m c_1\cdots c_{m-1})
    \end{equation}
    for every reduced word \(c_1\cdots c_m\) of length \(m\geq2\). Then \(T(xy)=T(yx)\) for all \(x,y\in A*B\).

    We now define a linear map \(\Phi_0\colon A*B\to D\), prove that it satisfies the two assumptions above, and then descend it to the desired left inverse of \(\jmath\).
    For a reduced word \(u\) of even length \(2r\), choose a cyclic rotation \(a_1b_1\cdots a_rb_r\) beginning in \(A_0\) and let \(\langle u\rangle\in W_r\) be the class of \((a_1\otimes b_1)\otimes\cdots\otimes(a_r\otimes b_r)\).
    Any two such rotations differ by a cyclic permutation of the \(r\) blocks, so \(\langle u\rangle\) is well-defined.
    We define \(\Phi_0\) recursively on the decomposable tensors in the normal form~\eqref{f_free_product_normal_form} as follows.
    Set
    \[
        \Phi_0(1):=q_A(\overline{1_A})=q_B(\overline{1_B}),
        \qquad
        \Phi_0(a):=q_A(\overline a),
        \qquad
        \Phi_0(b):=q_B(\overline b)
    \]
    for \(a\in A_0\) and \(b\in B_0\), and set \(\Phi_0(u):=\langle u\rangle\) for every reduced word \(u\) of even length.
    If \(u=c_1\cdots c_m\) has odd length \(m\geq3\), write \(c_m c_1=\sigma 1+h\), where \(\sigma\in\Bbbk\) and \(h\in A_0\cup B_0\), and set
    \begin{equation}\label{f_free_product_Phi_recursion}
        \Phi_0(u):=\sigma\Phi_0(c_2\cdots c_{m-1})+\Phi_0(h c_2\cdots c_{m-1}).
    \end{equation}
    Every word on the right has smaller length than \(u\), so the recursion terminates.
    The scalar and complementary parts of \(c_m c_1\) depend bilinearly on \(c_m\) and \(c_1\), and the map \(\Phi_0\colon A*B\to D\) is therefore linear.
    Since the recursion decreases the word length, \(\Phi_0(u)\) is a finite sum of elements of \(P_{A,B},W_1,\ldots,W_{\lfloor m/2\rfloor}\).

    We now verify the two assumptions above for \(T=\Phi_0\).
    Its restrictions to \(A\) and \(B\) factor through \(A_{\natural}\) and \(B_{\natural}\), respectively, and therefore vanish on commutators.
    If \(u\) is a reduced word of odd length, identity~\eqref{f_free_product_cyclic_invariance} is exactly the recursive definition~\eqref{f_free_product_Phi_recursion}.
    If \(u\) is a reduced word of even length, the same identity follows because \(\langle u\rangle\) is unchanged by cyclically rotating \(u\) by one letter.
    Thus~\eqref{f_free_product_cyclic_invariance} holds for every reduced word, and the preceding argument shows that \(\Phi_0\) vanishes on \([A*B,A*B]\).
    It therefore descends to a linear map
    \[
        \Phi\colon(A*B)_{\natural}\longrightarrow D.
    \]

    Obviously, \(\Phi\jmath=\operatorname{id}_D\), so \(\jmath\) is injective.
    Since \(\jmath\) is also surjective, it is an isomorphism, and~\eqref{f_free_product_cocenter_decomposition} follows.
\end{proof}

\begin{corollary}\label{cor_trace_separation_upgrade}
    Let \(A\) and \(B\) be \(\Bbbk\)-algebras with \(A\neq\Bbbk1_A\) and \(B\neq\Bbbk1_B\). If the traces of the finite-dimensional representations of \(A*B\) separate \((A*B)_{\natural}\), then \(A*B\) is RFD.
\end{corollary}
\begin{proof}
    By \cite[Theorem~2]{lichtman1983residual}, it suffices to prove that \(A\) and \(B\) are RFD. Suppose \(A\) is not RFD, and pick a nonzero \(a\in A\) annihilated by every finite-dimensional representation of \(A\).
    Any finite-dimensional representation sends \(1_A\) to \(\operatorname{id}\neq0\), so no nonzero scalar multiple of \(1_A\) is so annihilated, and \(a\notin\Bbbk1_A\).
    As \(B\neq\Bbbk1_B\), choose \(b\in B\setminus\Bbbk1_B\).
    Choose vector-space decompositions \(A=\Bbbk1_A\oplus A_0\) and \(B=\Bbbk1_B\oplus B_0\) such that \(a\in A_0\) and \(b\in B_0\).
    The isomorphism from Proposition~\ref{prop_free_product_cocenter_decomposition} sends \(\overline{ab}\in (A*B)_{\natural}\) to \(a\otimes b\in A_0\otimes B_0\), which is nonzero. Since traces of finite-dimensional representations of \(A*B\) separate \((A*B)_{\natural}\), there is a finite-dimensional representation \(\rho\) of \(A*B\) with \(\operatorname{tr}\rho(ab)\neq0\).
    But every finite-dimensional representation of \(A*B\) restricts to one of \(A\), so \(\rho(a)=0\) and \(\operatorname{tr}\rho(ab)=0\), a contradiction.
    Hence \(A\) is RFD, and symmetrically \(B\) is RFD.
\end{proof}

\section{Trace radicals of free products}\label{sec:trace_refinement}

\begin{lemma}\label{lemma_matrix_coefficient_separation}
    Let \(C\) be an RFD \(\Bbbk\)-algebra and let \(C=\Bbbk 1\oplus C_0\) be a vector-space decomposition. Let \(E\subseteq C_0\) be finite-dimensional and let \(0\neq\lambda\in E^*\) be a linear functional. Then there exist a finite-dimensional representation \(\pi\colon C\to\operatorname{End}_{\Bbbk}(V)\), a basis \(\beta\) of \(V\), and distinct vectors \(u,v\in\beta\) such that
    \begin{equation}\label{f_free_product_matrix_coefficient}
        v^*\bigl(\pi(x)u\bigr)=\lambda(x)\quad\text{for every }x\in E,
    \end{equation}
    where \(v^*\in V^*\) is the coordinate functional of \(\beta\) dual to \(v\).
\end{lemma}
\begin{proof}
    The matrix coefficients \(x\mapsto\varphi(\pi(x)u)\), where \(\pi\colon C\to\operatorname{End}_{\Bbbk}(V)\) runs over the finite-dimensional representations of \(C\) and \(u\in V\), \(\varphi\in V^*\), form a linear subspace \(M\subseteq C^*\), and \(M\) separates the elements of \(C\). Write \(W:=\Bbbk 1\oplus E\), a finite-dimensional subspace of \(C\). The restrictions to \(W\) of the elements of \(M\) form a linear subspace of \(W^*\) separating the elements of \(W\). A proper subspace of \(W^*\) would have a nonzero common kernel in \(W\), hence this subspace is all of \(W^*\). Extend \(\lambda\) to \(W\) by \(\lambda(1):=0\) and choose \(\pi\), \(u\), \(\varphi\) with \(\varphi(\pi(y)u)=\lambda(y)\) for every \(y\in W\). Then \(\varphi(\pi(x)u)=\lambda(x)\) for \(x\in E\), while \(\varphi(u)=\varphi(\pi(1)u)=\lambda(1)=0\). Since \(\lambda\neq0\), both \(u\) and \(\varphi\) are nonzero. Choose a basis \(\gamma\) of \(\ker\varphi\) containing \(u\), and choose \(v\in V\) with \(\varphi(v)=1\). Then \(\beta:=\gamma\sqcup\{v\}\) is a basis of \(V\) containing the distinct vectors \(u,v\), and its coordinate functional \(v^*\) equals \(\varphi\). This proves~\eqref{f_free_product_matrix_coefficient}.
\end{proof}

Lemma~\ref{lemma_matrix_coefficient_separation} realizes a functional on a finite-dimensional subspace of an algebra as a single matrix coefficient. We may depict identity~\eqref{f_free_product_matrix_coefficient} by the arrow
\[
    \begin{tikzcd}
        u \arrow[r, "\lambda", "\pi"'] & v
    \end{tikzcd}
\]
We will need a generalization of Lemma \ref{lemma_matrix_coefficient_separation} suitable for studying traces of finite-dimensional representations of the free product \(A*B\). Giving such a representation is equivalent to giving representations \(\rho_A\colon A\to\operatorname{End}_{\Bbbk}(V)\) and \(\rho_B\colon B\to\operatorname{End}_{\Bbbk}(V)\) on a common finite-dimensional vector space \(V\). Fixing a basis \(\mathcal{B}\) of \(V\), the trace of an alternating word \(a_1b_1\cdots a_rb_r\in A*B\), for \(a_j\in A\) and \(b_j\in B\), has the expansion
\[
    \operatorname{tr}\bigl(\rho_A(a_1)\rho_B(b_1)\cdots\rho_A(a_r)\rho_B(b_r)\bigr)
    =\sum_{e_1,\ldots,e_{2r}\in\mathcal{B}}
    \prod_{j=1}^{r}
    e_{2j-2}^*\bigl(\rho_A(a_j)e_{2j-1}\bigr)\,
    e_{2j-1}^*\bigl(\rho_B(b_j)e_{2j}\bigr),
    \qquad e_0:=e_{2r},
\]
where \(e_i^*\) denotes the coordinate functional dual to \(e_i\). Thus each summand can be represented as a closed chain of matrix coefficients. The right analog of identity \eqref{f_free_product_matrix_coefficient} should be a collection of similar identities, see Lemma \ref{lemma_two_algebra_coefficients_cyclic} below, arranged in a circle:
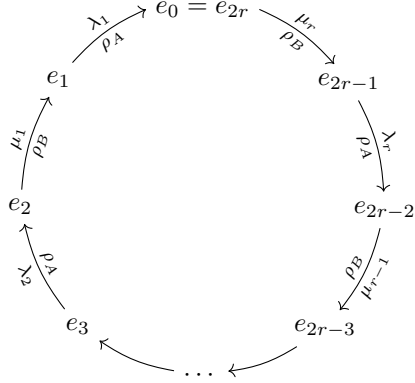
\[
    \begin{tikzpicture}[
        baseline=(current bounding box.center),
        cycle arrow/.style={draw, line width=rule_thickness, -{tikzcd to}},
        cycle label/.style={
            font=\everymath\expandafter{\the\everymath\scriptstyle},
            inner sep=2pt
        }
    ]
        \def\circleradius{2.42cm}
        \node (e2r) at (90:\circleradius) {$e_{0}=e_{2r}$};
        \node (e2rminusone) at (37:\circleradius) {$e_{2r-1}$};
        \node (e2rminustwo) at (-5.5:\circleradius) {$e_{2r-2}$};
        \node (e2rminusthree) at (-48:\circleradius) {$e_{2r-3}$};
        \node (dots) at (-90.5:\circleradius) {$\cdots$};
        \node (ethree) at (-133:\circleradius) {$e_3$};
        \node (etwo) at (-175.5:\circleradius) {$e_2$};
        \node (eone) at (142:\circleradius) {$e_1$};

        \draw[cycle arrow] (e2r.east) to[bend left=12]
            node[cycle label, sloped, above] {$\mu_r$}
            node[cycle label, sloped, below] {$\rho_B$}
            (e2rminusone);
        \draw[cycle arrow] (e2rminusone) to[bend left=12]
            node[cycle label, sloped, above] {$\lambda_r$}
            node[cycle label, sloped, below] {$\rho_A$}
            (e2rminustwo);
        \draw[cycle arrow] (e2rminustwo) to[bend left=12]
            node[cycle label, sloped, below] {$\mu_{r-1}$}
            node[cycle label, sloped, above] {$\rho_B$}
            (e2rminusthree);
        \draw[cycle arrow] (e2rminusthree) to[bend left=12] (dots);
        \draw[cycle arrow] (dots) to[bend left=12] (ethree);
        \draw[cycle arrow] (ethree) to[bend left=12]
            node[cycle label, sloped, below] {$\lambda_2$}
            node[cycle label, sloped, above] {$\rho_A$}
            (etwo);
        \draw[cycle arrow] (etwo) to[bend left=12]
            node[cycle label, sloped, above] {$\mu_1$}
            node[cycle label, sloped, below] {$\rho_B$}
            (eone);
        \draw[cycle arrow] (eone) to[bend left=12]
            node[cycle label, sloped, above] {$\lambda_1$}
            node[cycle label, sloped, below] {$\rho_A$}
            (e2r.west);
    \end{tikzpicture}
\]

\begin{lemma}\label{lemma_two_algebra_coefficients_cyclic}
    Let \(A\) and \(B\) be RFD \(\Bbbk\)-algebras. Fix \(r\geq1\), finite-dimensional subspaces \(E\subseteq A_0\) and \(F\subseteq B_0\), and nonzero functionals \(\lambda_1,\ldots,\lambda_r\in E^*\) and \(\mu_1,\ldots,\mu_r\in F^*\). There exist representations \(\rho_A\colon A\to\operatorname{End}_{\Bbbk}(V)\) and \(\rho_B\colon B\to\operatorname{End}_{\Bbbk}(V)\) on a common finite-dimensional vector space \(V\) with a basis \(\mathcal{B}\), subsets \(\mathcal{B}_1,\ldots,\mathcal{B}_{2r}\subseteq\mathcal{B}\), and distinct vectors \(e_1,\ldots,e_{2r}\in\mathcal{B}\) such that
    \(
        \mathcal{B}_1,\ldots,\mathcal{B}_{2r},
        \{e_1,\ldots,e_{2r}\}
    \)
    are pairwise disjoint and the following conditions hold:
    \begin{enumerate}[label=\textup{\arabic*)},leftmargin=5ex]
        \item With the cyclic convention \(e_0:=e_{2r}\), define
        \(
            V_i:=\operatorname{span}_{\Bbbk}
            \bigl(\mathcal{B}_i\cup\{e_{i-1},e_i\}\bigr)
        \)
        for \(1\leq i\leq 2r\). Then each \(V_{2j-1}\) is a subrepresentation of \(\rho_A\), and each \(V_{2j}\) is a subrepresentation of \(\rho_B\). Moreover,
        \[
            V_A:=\bigoplus_{j=1}^{r}V_{2j-1}
            \qquad\text{and}\qquad
            V_B:=\bigoplus_{j=1}^{r}V_{2j}
        \]
        are direct-summand subrepresentations of \(\rho_A\) and \(\rho_B\), respectively. They satisfy
        \(
            V_A\cap V_B=\bigoplus_{i=1}^{2r}\Bbbk e_i.
        \)
        \item With \(e_1^*,\ldots,e_{2r}^*\in V^*\) the coordinate functionals of \(\mathcal{B}\) dual to \(e_1,\ldots,e_{2r}\), and with the cyclic convention \(e_0^*:=e_{2r}^*\), one has, for \(1\leq j\leq r\),
        \begin{equation}\label{f_free_product_cyclic_entries}
            \begin{aligned}
                e_{2j-2}^*\bigl(\rho_A(x)e_{2j-1}\bigr)&=\lambda_j(x)\quad(x\in E),\\
                e_{2j-1}^*\bigl(\rho_B(x)e_{2j}\bigr)&=\mu_j(x)\quad(x\in F),
            \end{aligned}
        \end{equation}
        where for \(j=1\) the first equation reads \(e_{2r}^*\bigl(\rho_A(x)e_1\bigr)=\lambda_1(x)\).
    \end{enumerate}
\end{lemma}
\begin{proof}
    For each \(j\in\{1,\ldots,r\}\), Lemma~\ref{lemma_matrix_coefficient_separation} applied to \((E,\lambda_j)\) yields a representation \(\pi_A^j\colon A\to\operatorname{End}_{\Bbbk}(W_A^j)\), a basis \(\beta_A^j\) of \(W_A^j\), and distinct vectors \(u_j^A,v_j^A\in\beta_A^j\) with
    \[
        (v_j^A)^*\bigl(\pi_A^j(x)u_j^A\bigr)=\lambda_j(x)\quad(x\in E),
    \]
    and, applied to \((F,\mu_j)\), a representation \(\pi_B^j\colon B\to\operatorname{End}_{\Bbbk}(W_B^j)\), a basis \(\beta_B^j\) of \(W_B^j\), and distinct vectors \(u_j^B,v_j^B\in\beta_B^j\) with
    \[
        (v_j^B)^*\bigl(\pi_B^j(x)u_j^B\bigr)=\mu_j(x)\quad(x\in F).
    \]
    Put
    \[
        W_A:=\bigoplus_{j=1}^{r}W_A^j,\qquad \pi_A:=\bigoplus_{j=1}^{r}\pi_A^j,\qquad
        W_B:=\bigoplus_{j=1}^{r}W_B^j,\qquad \pi_B:=\bigoplus_{j=1}^{r}\pi_B^j,
    \]
    and, for \(C\in\{A,B\}\), let \(\beta_C:=\bigsqcup_{j=1}^{r}\beta_C^j\) and set \(d_C:=\dim W_C\). Since every \(\beta_C^j\) contains the distinct vectors \(u_j^C,v_j^C\), one has \(d_C\geq2r\). Set \(d:=d_Ad_B\).

    Take
    \[
        V:=W_A^{\oplus d_B},\qquad \rho_A:=\pi_A^{\oplus d_B},
    \]
    a representation of \(A\) of dimension \(d\). Let \(\mathcal{B}\) be the basis of \(V\) formed by the \(d_B\) copies of \(\beta_A\). Identify \(W_A\) with the first summand and \(\beta_A\) with the corresponding subset of \(\mathcal{B}\). Define
    \[
        e_{2j-1}:=u_j^A\quad(1\leq j\leq r),\qquad
        e_{2j-2}:=v_j^A\quad(2\leq j\leq r),\qquad
        e_{2r}:=v_1^A
    \]
    inside this first copy, and adopt the cyclic convention \(e_0:=e_{2r}\).

    The direct sum \(\pi_B^{\oplus d_A}\) is a representation of \(B\) of dimension \(d\). Identify \(W_B\) with the first summand of \(W_B^{\oplus d_A}\) and regard \(\beta_B\) as its basis. Equip every other summand with a copy of \(\beta_B\), and let \(\widehat{\beta}_B\) be the resulting basis of \(W_B^{\oplus d_A}\). Choose a bijection \(\widehat{\beta}_B\to\mathcal{B}\) which, on the distinguished vectors in the first summand, is given by
    \[
        v_j^B\longmapsto e_{2j-1},\qquad u_j^B\longmapsto e_{2j}\qquad(1\leq j\leq r),
    \]
    and which sends the remaining \(d_B-2r\) vectors of the first-copy basis \(\beta_B\) into \(\mathcal{B}\setminus\beta_A\). This is possible because
    \[
        \lvert\mathcal{B}\setminus\beta_A\rvert
        =d_A(d_B-1)\geq d_B-1\geq d_B-2r.
    \]
    After these values have been chosen, match all remaining basis vectors arbitrarily. Let \(\Phi_B\colon W_B^{\oplus d_A}\to V\) be the linear extension of this basis bijection, and set \(\rho_B:=\Phi_B\,\pi_B^{\oplus d_A}\,\Phi_B^{-1}\).

    For \(1\leq j\leq r\), define
    \[
        \begin{aligned}
            \mathcal{B}_{2j-1}&:=\beta_A^j\setminus\{e_{2j-2},e_{2j-1}\},\\
            \mathcal{B}_{2j}&:=\Phi_B\bigl(\beta_B^j\bigr)\setminus\{e_{2j-1},e_{2j}\}.
        \end{aligned}
    \]
    The sets \(\mathcal{B}_1,\ldots,\mathcal{B}_{2r},\{e_1,\ldots,e_{2r}\}\) are pairwise disjoint subsets of \(\mathcal{B}\): this is clear among the odd-indexed sets and among the even-indexed sets, while every odd-indexed \(\mathcal{B}_i\) lies in \(\beta_A\) and every even-indexed \(\mathcal{B}_i\) lies in \(\mathcal{B}\setminus\beta_A\) by the choice of \(\Phi_B\). The definition of the \(V_i\) then gives
    \[
        V_{2j-1}=W_A^j,\qquad V_{2j}=\Phi_B\bigl(W_B^j\bigr).
    \]
    Hence each odd-indexed \(V_i\) is a subrepresentation of \(\rho_A\), each even-indexed \(V_i\) is a subrepresentation of \(\rho_B\), and
    \[
        V_A=W_A,\qquad V_B=\Phi_B(W_B),
    \]
    where \(W_A\) and \(W_B\) are the first summands of \(W_A^{\oplus d_B}\) and \(W_B^{\oplus d_A}\), respectively. Their remaining summands provide invariant complements.

    Finally, fix \(j\in\{1,\ldots,r\}\). For \(x\in E\), the operator \(\rho_A(x)\) preserves the block \(W_A^j\) of the first copy, on which \(e_{2j-2}^*\) restricts to the coordinate functional \((v_j^A)^*\) of \(\beta_A^j\). Hence
    \[
        e_{2j-2}^*\bigl(\rho_A(x)e_{2j-1}\bigr)=(v_j^A)^*\bigl(\pi_A^j(x)u_j^A\bigr)=\lambda_j(x).
    \]
    For \(x\in F\), as \(\Phi_B^{-1}(e_{2j})=u_j^B\), one has \(\rho_B(x)e_{2j}=\Phi_B\bigl(\pi_B^{\oplus d_A}(x)u_j^B\bigr)\). The functional \(e_{2j-1}^*\Phi_B\) is the coordinate functional of \(\widehat{\beta}_B\) dual to \(v_j^B\), because \(\Phi_B(v_j^B)=e_{2j-1}\). On the block \(W_B^j\), it restricts to \((v_j^B)^*\). Therefore
    \[
        e_{2j-1}^*\bigl(\rho_B(x)e_{2j}\bigr)=(v_j^B)^*\bigl(\pi_B^j(x)u_j^B\bigr)=\mu_j(x),
    \]
    which gives the two identities of~\eqref{f_free_product_cyclic_entries}.
\end{proof}

\begin{proposition}\label{prop_cyclic_trace_detection}
    Let \(A\) and \(B\) be RFD \(\Bbbk\)-algebras, fix \(r\geq1\), and let \(h\in A*B\) have length at most \(2r\).
    Suppose the length-\(2r\) component of \(h\) is an element \(w\in(A_0\otimes B_0)^{\otimes r}\) whose image in \(\bigl((A_0\otimes B_0)^{\otimes r}\bigr)_{C_r}\) is nonzero.
    Then there exists a finite-dimensional representation \(\rho\colon A*B\to\operatorname{End}_{\Bbbk}(V)\) such that \(\operatorname{tr}\rho(h)\neq0\).
\end{proposition}
\begin{proof}
    Choose finite-dimensional subspaces \(E\subseteq A_0\) and \(F\subseteq B_0\) such that \(w\in W:=(E\otimes F)^{\otimes r}\).
    The inclusion \(W\hookrightarrow(A_0\otimes B_0)^{\otimes r}\) is \(C_r\)-equivariant and therefore induces a map on coinvariants that sends the class of \(w\) in \(W_{C_r}\) to its given nonzero class, so the class of \(w\) in \(W_{C_r}\) is nonzero.
    Hence there is an invariant functional \(\Omega\in(W^*)^{C_r}\) such that \(\Omega(w)\neq0\).
    Here we identified \((W_{C_r})^*\cong(W^*)^{C_r}\).
    The vector space \(W^*\) is spanned by the decomposable functionals
    \[
        (\lambda_1\otimes\mu_1)\otimes\cdots\otimes(\lambda_r\otimes\mu_r),
        \qquad \lambda_j\in E^*,\quad \mu_j\in F^*.
    \]
    The space of invariant functionals \((W^*)^{C_r}\) is spanned by cyclic symmetrizations of decomposable functionals:
    \[
        \sum_{s=0}^{r-1}(\lambda_{1+s}\otimes\mu_{1+s})\otimes\cdots\otimes(\lambda_{r+s}\otimes\mu_{r+s})
        \qquad(\text{indices mod }r).
    \]
    Since \(\Omega\) does not vanish on \(w\), at least one of these symmetrizations does not vanish on \(w\); fix one and denote it by \(\Lambda\).
    Write \(\lambda_1,\ldots,\lambda_r\in E^*\) and \(\mu_1,\ldots,\mu_r\in F^*\) for the functionals occurring in this chosen \(\Lambda\).
    They are all nonzero, since otherwise \(\Lambda\) would vanish identically.

    Apply Lemma~\ref{lemma_two_algebra_coefficients_cyclic} to \(E,F\) and these functionals, obtaining representations \(\rho_A\) and \(\rho_B\) on a common finite-dimensional vector space \(V=\Bbbk\mathcal{B}\), together with \(\mathcal{B}_1,\ldots,\mathcal{B}_{2r}\), \(V_1,\ldots,V_{2r}\), and \(e_1,\ldots,e_{2r}\) satisfying properties~(1) and~(2) of that lemma.
    Recall the cyclic convention \(e_0:=e_{2r}\).
    For \(t=(t_1,\ldots,t_{2r})\in(\Bbbk^\times)^{2r}\), define diagonal matrices \(g_A(t),g_B(t)\in\operatorname{GL}_{\Bbbk}(V)\) in the basis \(\mathcal{B}\) by
    \[
        \begin{aligned}
            g_A(t)(v)&:=
            \begin{cases}
                t_{2j-1}v,& v=e_{2j-2}\text{ for some }1\leq j\leq r,\\
                v,& \text{otherwise},
            \end{cases}\\
            g_B(t)(v)&:=
            \begin{cases}
                t_{2j}v,& v=e_{2j-1}\text{ for some }1\leq j\leq r,\\
                v,& \text{otherwise},
            \end{cases}
        \end{aligned}
        \qquad(v\in\mathcal{B}).
    \]
    The conjugated representations of \(A\) and \(B\) define a representation \(\rho_t\colon A*B\to\operatorname{End}_{\Bbbk}(V)\) by
    \[
        \rho_t(a)=g_A(t)\,\rho_A(a)\,g_A(t)^{-1}\quad(a\in A),\qquad \rho_t(b)=g_B(t)\,\rho_B(b)\,g_B(t)^{-1}\quad(b\in B).
    \]
    Set \(Q(t):=\operatorname{tr}\rho_t(h)\), a Laurent polynomial in \(t_1,\ldots,t_{2r}\).
    We now use property~(1) of Lemma~\ref{lemma_two_algebra_coefficients_cyclic} to extract the coefficient of \(t_1\cdots t_{2r}\) in \(Q\), and property~(2) to prove that this coefficient is nonzero.

    Expand \(h\) by~\eqref{f_free_product_normal_form}.
    The scalar component contributes only a constant to \(Q\) and hence does not contribute to the coefficient of \(t_1\cdots t_{2r}\).
    Every positive-length homogeneous component is a finite sum of reduced words \(y_1\cdots y_k\) with \(1\leq k\leq2r\), and by linearity it suffices to consider one such word.

    Expanding into matrix entries, one has
    \begin{equation}
        \operatorname{tr}\bigl(\rho_t(y_1\cdots y_k)\bigr)
        =\sum_{z_0,\ldots,z_{k-1}\in\mathcal{B}}\;\prod_{p=1}^{k}\bigl(\rho_t(y_p)\bigr)_{z_{p-1},z_p},
        \qquad z_k:=z_0.
    \end{equation}
    Since \(g_A(t)\) and \(g_B(t)\) are diagonal in \(\mathcal{B}\), conjugation multiplies the \((z_{p-1},z_p)\)-entry of \(\rho_A(y_p)\) or \(\rho_B(y_p)\), according as \(y_p\in A_0\) or \(y_p\in B_0\), by the ratio of the two corresponding diagonal entries, each of which is either \(1\) or one of \(t_1,\ldots,t_{2r}\). Thus each of the \(k\) factors supplies at most one numerator variable and at most one denominator variable, which means that the total degree is not greater than \(k\leq 2r\). Here, by the total degree of a Laurent monomial, we mean the sum of the exponents, which may be negative.
    The monomial \(t_1\cdots t_{2r}\) has total degree \(2r\), so a summand indexed by \((z_0,\ldots,z_{k-1})\) can contribute to its coefficient only if \(k=2r\), every factor supplies a numerator, and no factor supplies a denominator.
    Every word of length \(2r\) occurring in \(h\) belongs to the expansion of \(w\in W\), so its letters satisfy \(y_p\in E\) for odd \(p\) and \(y_p\in F\) for even \(p\).
    For odd \(p\), the \(p\)-th factor comes from \(\rho_A(y_p)\) and supplies a numerator variable only when \(z_{p-1}=e_{i-1}\) for some odd \(i\), in which case that variable is \(t_i\). For even \(p\), the analogous statement for \(\rho_B(y_p)\) holds with \(i\) even. Thus, for each \(p\), there is a unique \(i_p\in\{1,\ldots,2r\}\) such that
    \[
        z_{p-1}=e_{i_p-1},\qquad i_p\equiv p\pmod2\qquad(1\leq p\leq 2r).
    \]

    Next, we show that for every tuple \((z_0,\ldots,z_{2r-1})\) indexing a nonzero summand that contributes to the coefficient of \(t_1\cdots t_{2r}\), the associated indices satisfy \(i_{p+1}\equiv i_p+1\pmod{2r}\) for \(1\leq p\leq2r\), where \(i_{2r+1}:=i_1\).
    Fix \(p\), set \(\sigma_p:=\rho_A\) if \(p\) is odd and \(\sigma_p:=\rho_B\) if \(p\) is even, and consider the entry \(\bigl(\sigma_p(y_p)\bigr)_{e_{i_p-1},\,e_{i_{p+1}-1}}\).
    Temporarily set \(q:=i_{p+1}-1\pmod{2r}\) so that \(q\in\{1,\ldots,2r\}\).
    The relation \(i_{p+1}\equiv p+1\pmod2\) makes \(q\) have the same parity as \(p\), so \(V_q\) is \(\sigma_p\)-invariant by Lemma~\ref{lemma_two_algebra_coefficients_cyclic}.
    Hence
    \[
        \sigma_p(y_p)e_{i_{p+1}-1}\in V_q
        =\operatorname{span}_{\Bbbk}\bigl(\mathcal{B}_q\cup\{e_{q-1},e_q\}\bigr).
    \]
    The coordinate functional \(e_{i_p-1}^*\) vanishes on this subspace unless \(e_{i_p-1}\) is one of its two distinguished basis vectors.
    Thus the entry \(\bigl(\sigma_p(y_p)\bigr)_{e_{i_p-1},\,e_{i_{p+1}-1}}\) can be nonzero only if \(i_{p+1}\equiv i_p\) or \(i_p+1\pmod{2r}\).
    Since \(i_p\equiv p\pmod2\) and \(i_{p+1}\equiv p+1\pmod2\), the indices \(i_p\) and \(i_{p+1}\) have opposite parity, so the first possibility is excluded, and the claim follows.
    Consequently \(i_p\equiv i_1+p-1\pmod{2r}\) for \(1\leq p\leq2r\), so each index \(1,\ldots,2r\) occurs exactly once among \(i_1,\ldots,i_{2r}\).
    Since the factor at position \(p\) contributes the numerator \(t_{i_p}\) and no factor contributes a denominator, the Laurent monomial multiplying this matrix-entry product is \(\prod_{p=1}^{2r}t_{i_p}=t_1\cdots t_{2r}\).
    Moreover \(z_p=e_{i_{p+1}-1}=e_{i_p}\), so the entry at position \(p\) is \(e_{i_p-1}^*\bigl(\rho_A(y_p)e_{i_p}\bigr)\) for odd \(p\) and \(e_{i_p-1}^*\bigl(\rho_B(y_p)e_{i_p}\bigr)\) for even \(p\).

    The condition \(i_p\equiv p\pmod2\) makes \(i_1\) odd. In the first sum below, \(i_1\) ranges over the odd elements of \(\{1,\ldots,2r\}\), and \(i_p\equiv i_1+p-1\pmod{2r}\). Writing \(i_1=1+2s\), we obtain
    \[
        \begin{aligned}
            [t_1\cdots t_{2r}]\operatorname{tr}\bigl(\rho_t(y_1\cdots y_{2r})\bigr)
            ={}&\sum_{\substack{1\leq i_1\leq2r\\i_1\text{ odd}}}
            e_{i_1-1}^*\bigl(\rho_A(y_1)e_{i_1}\bigr)
            e_{i_2-1}^*\bigl(\rho_B(y_2)e_{i_2}\bigr)\cdots\\
            &\qquad\qquad\qquad\qquad
            e_{i_{2r-1}-1}^*\bigl(\rho_A(y_{2r-1})e_{i_{2r-1}}\bigr)
            e_{i_{2r}-1}^*\bigl(\rho_B(y_{2r})e_{i_{2r}}\bigr)\\
            ={}&\sum_{s=0}^{r-1}
            e_{2s}^*\bigl(\rho_A(y_1)e_{2s+1}\bigr)
            e_{2s+1}^*\bigl(\rho_B(y_2)e_{2s+2}\bigr)\cdots\\
            &\qquad\qquad\qquad\qquad
            e_{2(r+s)-2}^*\bigl(\rho_A(y_{2r-1})e_{2(r+s)-1}\bigr)
            e_{2(r+s)-1}^*\bigl(\rho_B(y_{2r})e_{2(r+s)}\bigr),
        \end{aligned}
    \]
    where the subscripts of the \(e_i\) are read modulo \(2r\).

    By property~(2) of Lemma~\ref{lemma_two_algebra_coefficients_cyclic}, the summand indexed by \(s\) equals
    \[
        \prod_{c=1}^{r}\lambda_{c+s}(y_{2c-1})\,\mu_{c+s}(y_{2c})
        =\Bigl(\bigotimes_{c=1}^{r}\bigl(\lambda_{c+s}\otimes\mu_{c+s}\bigr)\Bigr)(y_1\otimes\cdots\otimes y_{2r}),
    \]
    where the subscripts of \(\lambda_j\) and \(\mu_j\) are read modulo \(r\). Summing over \(s=0,\ldots,r-1\) gives \(\Lambda(y_1\otimes\cdots\otimes y_{2r})\), and by multilinearity the coefficient of \(t_1\cdots t_{2r}\) in \(Q\) equals \(\Lambda(w)\neq0\).
    Since \(\Bbbk\) has characteristic zero, it is infinite, so the nonzero Laurent polynomial \(Q\) does not vanish identically on \((\Bbbk^\times)^{2r}\).
    Choose \(t\in(\Bbbk^\times)^{2r}\) such that \(Q(t)\neq0\).
    Then \(\rho_t\) is the required representation, since \(\operatorname{tr}\rho_t(h)=Q(t)\neq0\).
\end{proof}

\subsection{Proofs of the main results}

\begin{proof}[Proof of Theorem~\ref{thm_trace_radical_free_product}]
    Since \(A\) and \(B\) are RFD, they admit finite-dimensional representations, so their unit classes are nonzero in their cocenters.
    Proposition~\ref{prop_free_product_cocenter_decomposition} therefore shows that \(\iota_A\) and \(\iota_B\) are injective and the intersection of their images is the one-dimensional vector space spanned by \(\overline{1_{A*B}}\).
    Every finite-dimensional representation of \(A*B\) restricts to finite-dimensional representations of \(A\) and \(B\), so the right-hand side of the asserted equality is contained in \(\operatorname{Rad}_{\mathrm{tr}}(A*B)\).
    Since no nonzero multiple of a unit class belongs to a trace radical, this sum is direct.

    Let \(\xi\in\operatorname{Rad}_{\mathrm{tr}}(A*B)\), and write \(\xi\) according to the direct-sum decomposition in Proposition~\ref{prop_free_product_cocenter_decomposition}.
    Suppose that the component of \(\xi\) in some cyclic-coinvariant summand is nonzero, and let \(r\geq1\) be the largest index for which its component in \(\bigl((A_0\otimes B_0)^{\otimes r}\bigr)_{C_r}\) is nonzero.
    Choose \(a\in A\), \(b\in B\), and tensors \(w_j\in(A_0\otimes B_0)^{\otimes j}\) for \(1\leq j\leq r\) representing the corresponding components of \(\xi\), and set \(h:=a+b+w_1+\cdots+w_r\in A*B\).
    Then \(\overline h=\xi\), the element \(h\) has length at most \(2r\), and its length-\(2r\) component is \(w_r\), whose image in \(\bigl((A_0\otimes B_0)^{\otimes r}\bigr)_{C_r}\) is nonzero.
    Proposition~\ref{prop_cyclic_trace_detection} therefore gives a finite-dimensional representation \(\rho\) of \(A*B\) such that \(\operatorname{tr}\rho(h)\neq0\), contradicting \(\xi\in\operatorname{Rad}_{\mathrm{tr}}(A*B)\).
    Therefore every cyclic-coinvariant component of \(\xi\) vanishes, and hence \(\xi=\iota_A(\overline a)+\iota_B(\overline b)\) for some \(a\in A\) and \(b\in B\).

    Fix a representation \(\beta_0\colon B\to\operatorname{End}_{\Bbbk}(W)\) and set \(q:=\dim W\).
    For any representation \(\alpha\colon A\to\operatorname{End}_{\Bbbk}(U)\), the actions \(x\mapsto\alpha(x)\otimes\operatorname{id}_W\) of \(A\) and \(y\mapsto\operatorname{id}_U\otimes\beta_0(y)\) of \(B\) define a representation of \(A*B\) on \(U\otimes W\).
    Since \(\xi\) belongs to the trace radical, one has
    \[
        q\operatorname{tr}\alpha(a)
        +(\dim U)\operatorname{tr}\beta_0(b)
        =0.
    \]
    Set \(c:=-q^{-1}\operatorname{tr}\beta_0(b)\).
    Then \(\operatorname{tr}\alpha(a-c1_A)=0\) for every \(\alpha\), and hence \(\overline{a-c1_A}\in\operatorname{Rad}_{\mathrm{tr}}(A)\).
    Fix a representation \(\alpha_0\colon A\to\operatorname{End}_{\Bbbk}(U_0)\), set \(p:=\dim U_0\), and let \(\beta\colon B\to\operatorname{End}_{\Bbbk}(W')\) be arbitrary.
    The preceding identity gives \(\operatorname{tr}\alpha_0(a)=pc\), while the corresponding representation of \(A*B\) on \(U_0\otimes W'\) gives
    \[
        0
        =(\dim W')\operatorname{tr}\alpha_0(a)
        +p\operatorname{tr}\beta(b)
        =p\operatorname{tr}\beta(b+c1_B).
    \]
    Thus \(\overline{b+c1_B}\in\operatorname{Rad}_{\mathrm{tr}}(B)\).
    Since \(\iota_A(\overline{1_A})=\iota_B(\overline{1_B})\), we conclude that
    \[
        \xi
        =\iota_A\bigl(\overline{a-c1_A}\bigr)
        +\iota_B\bigl(\overline{b+c1_B}\bigr),
    \]
    which belongs to the right-hand side of the asserted equality.
\end{proof}

\begin{proof}[Proof of Corollary~\ref{cor_finite_dimensional_free_products}]
    Every finite-dimensional algebra is RFD by its left regular representation.
    Since \(A_i\) is finite-dimensional, its Jacobson radical \(J(A_i)\) is nilpotent.
    Thus every element of \(J(A_i)\) acts nilpotently in every finite-dimensional representation of \(A_i\), and hence
    \(
        \overline{J(A_i)}
        \subseteq
        \operatorname{Rad}_{\mathrm{tr}}(A_i).
    \)

    Conversely, set \(S_i:=A_i/J(A_i)\).
    The kernel of the induced map \((A_i)_{\natural}\to(S_i)_{\natural}\) is \(\overline{J(A_i)}\).
    Since \(\Bbbk\) is algebraically closed, the Artin--Wedderburn theorem shows that \(S_i\) is a finite product of full matrix algebras over \(\Bbbk\).
    The ordinary traces on the matrix factors separate \((S_i)_{\natural}\).
    Therefore every class in \((A_i)_{\natural}\) not belonging to \(\overline{J(A_i)}\) is detected by the trace of a finite-dimensional representation of \(A_i\), and hence
    \(
        \operatorname{Rad}_{\mathrm{tr}}(A_i)
        =
        \overline{J(A_i)}.
    \)

    By repeated applications of \cite[Theorem~2]{lichtman1983residual}, all partial free products are RFD, and repeated applications of Theorem~\ref{thm_trace_radical_free_product} now give the desired result.
\end{proof}

We conclude with an application to group algebras. See \cite[Proposition~3.1]{BassLubotzky1983} for a companion statement concerning individual group elements, relating equality of all finite-dimensional character values to conjugacy in every finite quotient.

\begin{proposition}\label{prop_group_algebra}
    Let \(G\) be a group.
    \begin{enumerate}[label=\textup{(\roman*)},leftmargin=5ex]
        \item If \(G\) is residually finite, then \(\Bbbk[G]\) is RFD.
        \item If \(\Bbbk\) is algebraically closed and \(G\) is conjugacy separable, then \(\Bbbk[G]\) is trace-RFD.
    \end{enumerate}
\end{proposition}
\begin{proof}
    \emph{(i)} Let \(0\neq a\in\Bbbk[G]\) and write \(a=\sum_{s=1}^{r}c_s g_s\) with \(g_1,\ldots,g_r\) pairwise distinct elements of \(G\) and all \(c_s\neq0\). For \(s\neq t\) one has \(g_s g_t^{-1}\neq1\), so residual finiteness yields a finite quotient of \(G\) in which \(g_s\) and \(g_t\) have distinct images. Let \(q\colon G\to Q\) be the diagonal map into the product of these finitely many quotients, so that \(Q\) is finite and \(q(g_1),\ldots,q(g_r)\) are pairwise distinct. Let \(\rho\colon\Bbbk[G]\to\operatorname{End}_{\Bbbk}(\Bbbk[Q])\) be the left regular representation of \(Q\) precomposed with \(q\). Evaluating at the basis vector \(1_Q\in\Bbbk[Q]\),
    \[
        \rho(a)(1_Q)=\sum_{s=1}^{r}c_s\,q(g_s)\neq0,
    \]
    since the \(q(g_s)\) are distinct basis vectors of \(\Bbbk[Q]\) and the \(c_s\) are nonzero. Hence \(\rho(a)\neq0\), and \(\Bbbk[G]\) is RFD.

    \emph{(ii)} Let \(0\neq\xi\in\Bbbk[G]_{\natural}\). Since the conjugacy classes form a basis of \(\Bbbk[G]_{\natural}\), we may choose a representative \(a=\sum_{s=1}^{r}c_s g_s\in\Bbbk[G]\) with \(g_1,\ldots,g_r\) in pairwise distinct conjugacy classes and all \(c_s\neq0\). Conjugacy separability implies that finite-dimensional characters of \(G\) separate its conjugacy classes: if \(g,h\in G\) are non-conjugate, choose a finite quotient in which their images are non-conjugate. Since the irreducible characters of a finite group form a basis of its class functions, some character of this quotient takes different values on the two images. Pulling this character back along the quotient map gives a finite-dimensional character of \(G\) that distinguishes \(g\) and \(h\).

    Let \(\mathcal{C}\subseteq\Bbbk^r\) be spanned by the vectors \(\bigl(\chi(g_1),\ldots,\chi(g_r)\bigr)\), where \(\chi\) ranges over the finite-dimensional characters of \(G\). The trivial character gives \(\mathbf{1}:=(1,\ldots,1)\in\mathcal{C}\), and the identity \(\chi_{V\otimes W}=\chi_V\chi_W\) shows that \(\mathcal{C}\) is a unital subalgebra of \(\Bbbk^r\). For \(s\neq t\), the preceding paragraph gives \(f^{s,t}=(f_1^{s,t},\ldots,f_r^{s,t})\in\mathcal{C}\) with \(f_s^{s,t}\neq f_t^{s,t}\). Hence \(\mathcal{C}\) contains
    \[
        u^{s,t}:=
        \frac{f^{s,t}-f_t^{s,t}\mathbf{1}}
        {f_s^{s,t}-f_t^{s,t}},
    \]
    whose \(s\)-th entry is \(1\) and whose \(t\)-th entry is \(0\). For each \(s\), the product \(e_s:=\prod_{t\neq s}u^{s,t}\) belongs to \(\mathcal{C}\); its \(s\)-th entry is \(1\), while its \(t\)-th entry is \(0\) for every \(t\neq s\). Thus \(e_1,\ldots,e_r\) are the standard basis of \(\Bbbk^r\), so \(\mathcal{C}=\Bbbk^r\). The linear functional \((z_1,\ldots,z_r)\mapsto\sum_{s=1}^{r}c_s z_s\) is nonzero. Consequently, it is nonzero on a linear combination of the vectors spanning \(\Bbbk^r=\mathcal{C}\), and hence on \(\bigl(\chi(g_1),\ldots,\chi(g_r)\bigr)\) for some finite-dimensional character \(\chi\) of \(G\), which means that \(\chi(\xi)\neq 0\). Finally, a conjugacy separable group is residually finite, so \(\Bbbk[G]\) is RFD by~(i), and \(\Bbbk[G]\) is trace-RFD.
\end{proof}

\printbibliography

\end{document}